\documentclass{article} 
\usepackage{amssymb} 
\usepackage{epstopdf} 
\usepackage{amsmath} 
\usepackage{amsthm} 
\usepackage{mathrsfs} 
\newtheorem{theorem}{Theorem}[section] 
\newtheorem*{theorem_1}{Theorem A} 
\newtheorem*{theorem_2}{Theorem B} 
\newtheorem{lemma}[theorem]{Lemma} 
\newtheorem{corollary}[theorem]{Corollary} 
\newtheorem{proposition}[theorem]{Proposition} 
\newtheorem{definition}[theorem]{Definition} 
\newtheorem{remark}[theorem]{Remark} 
 
\DeclareMathOperator{\rk}{rk} 
\DeclareMathOperator{\gk}{\mathcal{GK}} 
\DeclareMathOperator{\ann}{ann} 
\DeclareMathOperator{\s}{\ast} 
\title{A note on modules over Quantum Laurent Polynomials} 
\author{Ashish Gupta} 
 
\begin{document} 
\maketitle 
\begin{abstract}   
A twisted group  
algebra $F \s A$ of a finitely generated free abelian group $A$ over a  
field $F$ is in  
general noncommutative, but $A$ may contain nontrivial subgroups $C$ so that  
the subalgebra $F \s C$ is commutative. In this paper, we show that  
an $F \s A$-module $M$ which is finitely generated as an $F \s C$-module for any commutative  
subalgebra $F \s C$ must be artinian. We also show that $M$ must be torsion-free as $F \s C$-module and the  
Gelfand--Kirillov dimension of $M$ must equal the rank of $C$.  
We then apply these results to modules  
over finitely  
generated nilpotent groups  
of class $2$.  
\end{abstract} 
 
\section{Introduction} 
A quantum Laurent polynomial algebra $P(\mathfrak q, F)$ 
over a field $F$ 
is defined as the associative  
$F$-algebra generated over  
$F$ by the variables
$u_1, \cdots, u_n$ and their  
inverses, satisfying the relations,  
\[ u_iu_j = q_{ij}u_ju_i, \]  
where $q_{ij} \in F$ are nonzero  
scalars and $\mathfrak q$ is the $n \times n$ matrix $(q_{ij})$. The $q_{ij}$ satisfy  
\[ q_{ii} = 1 = q_{ij}q_{ji}, \ \ \ \ \ \ \ i,j = 1, \cdots, n. \]   
These algebras have been  
called by various names.  They have been called the \emph{multiplicative analouges of the Weyl algebra}, the \emph{McConnell--Pettit algebras}, the \emph{quantum Laurent polynomial algebras} and the \emph{quantum tori}. They arise as localizations of  
group  
algebras (see \cite{Br}) and also  
play an important role in noncommutative geometry (see \cite{M}).  
 
The quantum Laurent polynomial algebras are  
precisely the twisted group algebras (see \cite[Chapter 1]{P1})  
$F \s A$ of a finitely generated free abelian group $A$ over a  
field $F$. In this paper we shall mainly use the notation  
$F \s A$ for a quantum Laurent polynomial algebra. Sometimes 
the following more explicit notation shall be needed:
\[ F_{\mathfrak q}[u_1^{\pm 1}, \cdots, u_n^{\pm 1}], \]
where $\mathfrak q = (q_{ij})$.

Subsequently the situation with an arbitrary number of generators $u_1^{\pm 1}, \cdots, u_n^{\pm 1}$ was studied in 
\cite{MP}. In \cite{MP}, it was noted that the Krull and the global dimensions of $F \s A$ must coincide. 
An interesting and important criterion was given for these dimensions in \cite[Corollary 3.8]{MP} in terms of certain 
localizations (see \cite[]{MP}). \\  In \cite{AG}.
Modules over quantum Laurent polynomials have been considered in \cite{A1}, \cite{A2}, \cite{BG1}, \cite{BG2} and \cite{MP}.         
In this note we employ the methods of \cite{AG} and \cite{MP} and the work in \cite{BG2} in order to prove the following theorem for modules over $F \s A$:

\begin{theorem_1} 
\label{MT1} 
Let $M$ be a nonzero finitely  
generated  $F \s A$-module, where  
$F \s A$ has center $F$.  
Let $C < A$ be a subgroup which contains a subgroup $C_0$ of  
finite index such that $F \s C_0$ is commutative.  
If $M$ is finitely generated as an $F \s C$-module then,    
\begin{enumerate} 
\item[(i)] $\gk(M) = \rk(C)$,  
\item[(ii)] $M$ is $F \s C$-torsion-free,  
\item[(iii)] $M$ is artinian,  
\item[(iv)] $M$ is cyclic.      
\end{enumerate}  
\end{theorem_1}  

In this theorem $\gk(M)$ denotes the Gelfand--Kirillov dimension of $M$.
It follows from this  
result that each finitely  
generated $F \s A$-module  
with GK dimension one must be  
artinian (see Corollary 2.4). 
   
Artinian modules  
arise in other  
situations also. For example it was 
shown in \cite{Ar} that if the  
multiparameters $q_{ij}$, where  
$1 \le i < j \le n$, are 
independent in $F^*$, then each  
finitely generated periodic  
$F \s A$-module is  
artinian and cyclic. \\ 
 
In Proposition \ref{app_nilp_gr} 
we give an  
application of this theorem  
to finitely generated modules over finitely generated and  
torsion-free nilpotent groups $H$ of class 2. 
If $k$ is field we denote by $kH$ the  
group algebra of $H$ over $k$. 
A $kH$-module $M$ is   
said to be \emph{reduced} if the annihilator of $M$ in $k\zeta H$  
is a prime ideal $P$  
and $M$ is  
$k\zeta H/P$-torsion-free.  
Every finitely generated kH-module is  
filtered by a finite series  
each section of which is reduced  
with some prime ideal of $k\zeta H$. 
The following theorem which is a special case of is deduced  
\begin{theorem_2} 
Let $H$ be a finitely generated  
torsion-free nilpotent group of class $2$. 
A finitely generated  
$kH$-module $Q$ which is  
reduced with annihilator $P$, and  
is finitely generated over $kL$ for some  
abelian subgroup $L \le H$ is torsion-free as $kL/P.kL$-module. Moreover if  
$\zeta H$ is cyclic then $Q$ is also  
artinian.          
\end{theorem_2} 
 
\section{Background on $F \s A$} 
 
The twisted group algebra  
$F \s A$ of a finitely generated 
free abelian group $A$ over a field  
$F$ has as an $F$-basis a copy  
$\overline A : = \{ \bar a, a \in A \}$ of $A$. 
The multiplication of  
basis elements $\bar a_1, \bar a_2$ is defined  by  
\[ \bar a_1 \bar a_2 =  
\lambda(a_1,a_2)\overline{a_1a_2}, \]  
where $\lambda: A \times A \rightarrow F^*$ is a function. 
The associativity  
condition requires that the map  
$\lambda$ be a $2$-cocycle (see \cite[Chapter 1]{P1}). 
The scalars $\kappa \in F$  
commute with the basis elements. Thus $\kappa \bar a = \bar a \kappa$ for $\kappa \in F$ and $a \in A$.  Thus any $\alpha \in F \s A$ can be uniquely  
expressed in the form  
$\alpha =  
\sum_{a \in A}\lambda_a \bar a$,  
where $\lambda_a = 0$ for all but finitely many  
$a \in A$. The subset of elements $a \in A$ such that  
$\lambda_a \ne 0$ is known as the \emph{support  
of $\alpha$ in $A$}.  
For a subgroup $B$ of $A$, the subset of elements 
$\beta \in F \s A$ such that the  
support of $\beta$ lies in  
$B$, has the structure of a  
twisted group algebra of $B$ over  
$F$. This is denoted as $F \s B$. 
As already mentioned, $F \s A$ has as an $F$-basis a copy 
$\{ \bar a, a \in A \}$ of $A$. By a 
\emph{diagonal change in basis} we may also take the subset $\{ \kappa_a \bar a, a \in A, \kappa_a \in F^*\}$ as an $F$-basis for $F \s A$. As explained in \cite[Chapter 1]{P1}, with the help of a diagonal change in basis it can be shown that $F \s Z$ is commutative whenever $Z$ is an infinite cyclic subgroup of $A$.
By \cite[Theorem A]{Br}, the maximal rank of a 
subgroup $B$ of $A$ such that $F \s B$ is commutative is equal to the Krull and the global dimension of $F \s A$.  
Examples of $F \s A$ for which the commutativity of $F \s B$ implies that $B$ is an infinite cyclic subgroup of $A$ can be found in \cite{MP}.

It is known (e.g. \cite[Lemma 37.8]{P2}) 
that for each subgroup $B \le A$, the monoid of nonzero elements of $F \s B$ is an Ore subset of $F \s A$.
Thus we may localize $F \s A$ at $X: = F \s B \setminus \{0\}$. 
This localization $(F \s A)X^{-1}$ is itself a crossed product (see \cite[Chapter 1]{P2}) $D \s A/B$ of the group $A/B$ over the division ring $D$ of 
quotients of $F \s B$. In this connection we shall adopt a useful notation that was introduced in \cite{MP}: 
let $\{x_1, \cdots, x_n \}$ be a basis for $A$ such that $\{x_1, \cdots, x_k\}$ is basis for $B$, 
where $k \le n - 1$. 
Then $(F \s A)X^{-1}$ is denoted as:
\[ F(x_1, \cdots, x_k)[x_{k +1}, \cdots, x_n], \] 
where the localised generators are enclosed in roundparenthesess.\\
As shown in \cite{MP}, 
the unit group $\mathcal U$ of $F \s A$ is the set 
\[ \mathcal U = \{ \kappa  \bar a , \kappa \in F^*, a \in A \}. \]
The derived subgroup $\mathcal U'$ of $\mathcal U$ is a subgroup of $F^*$. Thus 
$\mathcal U$ is nilpotent of class at most $2$. For subsets $X$ and $Y$ of $A$, we denote by $\overline X$ the subset 
$\{ \bar x, x \in X\}$ of $\mathcal U$, and by $[ \overline X, \overline Y ]$ the subgroup of $\mathcal U'$ generated by the commutators $[\bar x, \bar y]$, where $x \in X$ and $y \in Y$.  
Since $\mathcal U'$ is central in $\mathcal U$, hence by \cite[Chapter 5]{Ro},
\begin{align*} 
[\bar x_1\bar x_2, \bar y] &= [\bar x_1, \bar y][\bar x_2, \bar y] \\ 
[\bar x, \bar y_1 \bar y_2] &= [\bar x, \bar y_1][\bar x, \bar y_2] 
\end{align*}
A dimension for studying 
modules over a crossed product 
$D \s A$ of a free finitely  
generated abelian group $A$ over a division ring $D$, denoted as $\dim(M)$, 
was introduced in \cite{BG2} 
and was shown to coincide with the Gelfand--Kirillov dimension (\cite{KL}).   
  
\begin{definition}[Definition 2.1 of \cite{BG2}] 
\label{def_dim} 
Let $M$ be a $D \s A$-module. The dimension $\dim (M)$ of 
$M$ 
is defined to be 
the greatest integer $r$, $0 \le r \le \rk(A)$, so that 
for some subgroup $B$ in 
$A$ with rank $r$,  
$M$ is not $D \s B$-torsion. 
\end{definition}
 
\begin{proposition}[Lemma 2.2(1) of \cite{BG2}] 
\label{dim_is_exact} 
Let \[ 
0 \rightarrow M_1 \rightarrow M \rightarrow M_2 \rightarrow 0, \] 
be an exact sequence of $D \s A$-modules. Then, 
\[ \dim(M) = \sup(\dim(M_1), \dim(M_2)). \] 
\end{proposition}

\section{The proofs of Theorems A and B}

The following lemma shall be  
key tool in the proof of  
Theorem A.    
\begin{lemma} 
\label{AG_reslt} 
Suppose that $F \s A$ has a 
finitely generated module 
$M$ and $A$ has a subgroup $C$  
with $A/C$ torsion free, 
$\rk(C) = \gk(M)$, 
and $F \s C$ commutative. Suppose moreover that $M$ 
is not $F \s C$-torsion. 
Then $C$ has a virtual complement $E$ in $A$ 
such that $F \s E$ is commutative.  
In fact given $\mathbb Z$-bases  
$ \{x_1,\cdots, x_r \}$ and $\{ x_1,\cdots, x_r,  
x_{r + 1}, \cdots, x_n \}$ for  
$C$ and $A$ respectively, there exist monomials $\mu_j, j =  
r + 1, \cdots n,$  
in $F \s C$,  
and an integer $s > 0$ such that the monomials 
$\mu_j \bar{x_j}^s$ commute in $F \s A$. 
\end{lemma}
 
\begin{proof} 
Let  
$\bar {x_i}\bar{x_j} = q_{ij}\bar{x_j}\bar{x_i}$, where, $i,j = 1, \cdots, n$ and $q_{ij} \in F^*$.    
We set $S = F \s C\setminus\{0\}$ and denote the 
quotient field $(F \s C)S^{-1}$ by $F_S$. 
Then $(F \s A)S^{-1}$ is a crossed product 
\[ R  = F(x_1, \cdots, x_r)[ x_{r + 1}, \cdots, x_{n}]. \] 
The corresponding module of fractions 
$MS^{-1}$ is nonzero as $M$ 
is not $S$-torsion by the hypothesis. 
Furthermore, by the hypothesis,  
$\gk(M)  = \rk(C)$ and so in view of 
\cite[Lemma 2.3]{BG2},  
$MS^{-1}$ is finite dimensional as $F_S$-space. 
It is shown in \cite[Section 3]{AG} that if 
$R$ has a module that is one dimensional 
over $F_S$ 
then there exist monomials 
$\mu_i \in F \s C$ 
such that the monomials $\mu_i \bar x_i$, where $r + 1 \le i \le n$, 
commute mutually. Since the scalars in $F$ are central in $F \s A$, we may assume that 
$\mu_i \in \overline {C}$. 
If $c_i$ is the element of $C$ corresponding to $\mu_i$, we may in this case take 
\[ E = \langle c_{r + 1}x_{r + 1}, \cdots, c_nx_n \rangle. \]
the $s$-fold exterior power $M': = \wedge^s(MS^{-1}_{F_S})$, where, 
$s = \dim_{F_S} (MS^{-1})$ is a 
one dimensional module over the crossed product $R'$ obtained from $R$ by 
raising the cocycle of $R$ to its $s$-th power. More explicitly $R'$ is the crossed product 
\[ R'  = R  = F(x_1, \cdots, x_r)[ x_{r + 1}, \cdots, x_{n}] \] 
but with $[\bar x_i, \bar x_j] = q'_{ij}$, where
\begin{equation}  
\label{s_th_power_of_cocyle} 
q'_{ij} = q_{ij} , \ \ \  1 \le i \le  r;   \ \ \ \ \ \ \  
q'_{ij} = q_{ij}^s, \ \ \  r < i, j \le n.  
\end{equation} 
Then there are monomials $\mu_j \in F \s C$ such that 
the monomials 
$\mu_j \bar x_j$, where $j = r +1, \cdots, n$ commute with respect to the cocycle defined in equation \ref{s_th_power_of_cocyle}.
It easily follows from this that the monomials 
$\{\mu_j \bar x_j^s\}_{j = r+1}^n$ commute in $F \s A$. 
\end{proof} 
 
\begin{remark}
\label{fi}
The above lemma remains vaild if we drop 
the assumption that $A/C$ is torsion-free. In this case $A$ has a subgroup $A'$
of finite index such that $A'/C$ is torsion-free.  
Moreover in view of \cite[Lemma 2.7]{BG2}, $M$ is finitely generated and has GK dimension 
equal to $\rk(C)$ as $F \s A'$-module.  

\end{remark}

We are now ready to prove, 
\begin{theorem_1} 
\label{MT1} 
Let $M$ be a nonzero finitely  
generated  $F \s A$-module where  
$F \s A$ has center $F$.  
Let $C < A$ be a subgroup having a subgroup $C_0$ of  
finite index such that $F \s C_0$ is commutative.  
If $M$ is finitely generated as an $F \s C$-module then,    
\begin{enumerate} 
\item[(i)] $\gk(M) = \rk(C)$,  
\item[(ii)] $M$ is $F \s C$-torsion-free,  
\item[(iii)] $M$ is artinian,  
\item[(iv)] $M$ is cyclic.      
\end{enumerate}  
\end{theorem_1}  
 
\begin{proof} (i) 
We shall denote the module 
$M$ regarded as 
$F \s C$-module as $M_C$. 
By hypothesis $M_C$ is finitely generated and 
so $M_{C_0}$ is also finitely generated where $M_{C_0}$ denotes $M_C$ viewed as $F \s C_0$-module.
By \cite[Lemma 2.7]{BG2}, \[ \gk(M) = \gk(M_C) = \gk(M_{C_0}) \].
If $\gk(M) < \rk(C)$ we may pick a subgroup $E_0 < C_0$ with 
$\rk(E_0) < \rk(C_0)$ such that $M_{C_0}$ is not 
$F \s E_0$-torsion and $\gk(M) = \rk(E_0)$ (section \ref{}). 
Moreover $F \s E_0$ is commutative 
since $F \s C_0$ is commutative by the hypothesis.
By Lemmma \ref{AG_reslt} and 
Remark \ref{fi}, 
$E_0$ has 
a (virtual) complement  
$E_1$ in $A$ such that $F \s E_1$ 
is commutative.
Since $\rk(E_1) + \rk(C_0)$ exceeds $\rk(A)$, Hence $E_1 \cap C_0 > \langle 1 \rangle$.
Moreover as $E_1E_0$ has finite index in $A$, hence $E_1C_0$ has finite index in $A$.
But $\overline{E_1 \cap C_0}$ is central in  $F \s E_1C_0$ and hence $F \s A$ has center larger than $F$.
This is contrary to the hyopthesis in the theorem.
 
(ii) Suppose that the $F \s C$-torsion submodule $T$ of $M$ is nonzero.  
We recall that $T$ is an $F \s A$-submodule of $M$.    
Applying part (i) of the theorem just established to $T$  
we obtain $\gk(T) = \rk(C)$. 
We shall denote the $F \s C$-module structure on $T$ by $T_C$.
We note that $T_C$ is finitely generated. Moreover 
\[ \gk(T_C) = \rk(T) = \rk(C) \] 
by \cite[Lemma 2.7]{BG2}. 
It then follows from \cite[Proposition 2.6]{BG2}, 
that there exists $t \in T$ so that $\ann_{F \s C}(t) = 0$.
But by the definition of $T$ this is $\ann_{F \s C}(t) \not = 0$.
Hence $T = 0$ and $M$ is $F \s C$-torsion-free. 

(iii) We first note that each nonzero subfactor of $M$ has the same GK dimension as $M$. Let $N$ and $L$ be submodules of $M$ such that $Q: = N/L$ is nonzero. 
As $Q$ finitely generated over 
$F \s C$ it is $F \s C$-torsion free by part (ii) of the theorem shown above.  
It follows from Definition \ref{def_dim} and Proposition \ref{dim_is_exact} that 
\[ \gk(Q) = \rk(C) = \gk(M). \]
Now let
\begin{equation} 
\label{str_des_seq_1} 
 M = M_0 > M_1 > M_2 > \cdots,   
\end{equation} 
be a strictly descending  
sequence of submodules  
of $M$. By the above $\gk(M_i/M_{i + 1}) = \gk(M)$ for all $i \ge 0$.
By \cite[Lemma 5.6]{MP} and \cite[5.9]{MP} the sequence  
(\ref{str_des_seq_1}) must become constant after a finite number of  
steps. Hence $M$ is artinian.     
 
(iv) Since $F \s A$ has center  
$F$, it is simple by  
\cite[Proposition 1.3]{MP}.  
It follows  
from  
\cite[Corollary 1.5]{Ba2} that $M$ is cyclic.   
\end{proof} 
\begin{corollary} 
\label{dim_1_mod} 
Each finitely generated module with Gelfand--Kirillov  
dimension less than or equal to one over an $F \s A$  
with center exactly $F$ is  
artinian and cyclic.   
\end{corollary} 
\begin{proof} 
We recall that an $F \s A$-module with GK dimension zero is finite  
dimensional (e.g. \cite{KL}) over the base field $F$ and so is artinian. 
By \cite[Lemma 2.9]{Wa2}, an $F \s A$-module with GK dimension one is finitely generated over some  
$F \s C$ with $C$ infinite cyclic. Since $F \s C$ must be commutative, the result follows from Theorem A. 
\end{proof} 
 
\begin{proposition} 
\label{hol_mod} 
A finitely generated $F \s A$-module $M$ 
satisfying \[ \gk(N) + \dim(F \s A)  = \rk(A) \]  
has finite length.    
\end{proposition} 

\begin{proof} 
By \cite[Theorem A]{B}, $\dim(F \s A)$ is equal to the maximal rank of a subgroup 
$B$ of $A$ such that $F \s B$ is commutative. Hence in view of 
\cite{Theorem 3}{B}, the minimum possible GK dimension of a nonzero finitely generated 
$F \s A$-module is $\rk(A) - \dim(F \s A) = \gk(M)$.  
Thus in any (strictly) descending 
sequence  $M = M_0 > M_1 > \cdots >$ of submodules of $M$, $\gk(M_i/M_{i + 1}) = \gk(M)$ for each $i \ge 0$. It then p follows from \cite[Lemma 5.6]{MP} and  
\cite[5.9]{MP} that this sequence must halt. Hence $M$ has finite length.
\end{proof} 
 
We now investigate the situation in which a submdoule of a finitely generated $F \s A$-module is finitely generated 
over a commutative subalgebra $F \s C$ for $C < A$.
\begin{definition}
A nonzero $F \s A$-module $N$ is called \emph{critical} when $N/L$ has strictly smaller GK dimension than $N$ 
for each nonzero proper submodule $L$ of $N$. 
\end{definition}
 
\subsection{Applications to nilpotent groups} 
We shall now consider some applications to finitely generated  
torsion-free nilpotent groups of class $2$. 
Throughout this section $H$ stands for such a group and  
$k$ denotes a field. We denote the group algebra of $H$ over $k$ by  
$kH$ and by $\zeta H$ the center of $H$.    
   
In \cite{G}, the following definition is given. 
\begin{definition} 
A finitely generated $kH$-module  
$M$ is said to be \emph{reduced} if there is a prime ideal $P$ of  
$k\zeta H$ such that $M$ is annihilated by  
$P$ and is $k\zeta H/P$-torsion-free. 
\end{definition}    
 
It is remarked in \cite{G} that every finitely  
generated $kH$-module is filtered by a  
finite series in which each section is reduced. 
We are now ready to prove  
 
\begin{proposition} 
\label{app_nilp_gr} 
Let $H$ be a finitely generated  
torsion-free nilpotent group of class  
$2$. Let $M$ be a finitely generated  
$kH$-module reduced with annihilator $P$  
in $k\zeta H$.  
If $M$ is finitely generated as a 
$kL$-module for some abelian  
subgroup $L < H$ then, 
\begin{enumerate} 
\item[(i)] $M$ is torsion-free as  
$kL/P.kL$-module,
\item[(ii)] for any descending  
sequence  
\[  M_1 \ge M_2 \ge  \cdots \ge  \] of submodules of $M$, 
 there is an integer $s$ such that  
$M_k/M_{k + 1}$ is $k\zeta H/P$-torsion  
for all $k \ge s$, 
\end{enumerate}          
\end{proposition} 

\begin{proof} 
We note that there is no harm in assuming that $L > \zeta H$.
As in \cite{G}, let $K$ denote the field of fractions  
of $k\zeta H/P$ and set  
\[ \widehat{M} = M \otimes_{k\zeta H/P} K. \]  
Then $\widehat{M}$ is a finitely generated  
module for $kH/P.kH \otimes_{k\zeta H/P} K$ and the latter is a  
twisted group algebra  
$K \s A$, where $A = H/\zeta H$. 
Moreover the center of  
$K \s A$ is exactly $K$. 
By the hypothesis in the theorem, 
$M$ is finitely generated  
over $kL$ and so over  
$kL/P.kL$. Thus  
$\widehat{M}$ is finitely generated  
over $kL/P.kL \otimes_{k\zeta H/P} K$ which is the subalgebra $K \s C$ for  
$C = L/\zeta H$.   
Moreover since $L$ is abelian,  
therefore, $K \s C$ is commutative. 
By  Theorem A, $\widehat{M}$ is  $K \s C$-torsion-free. In other words $\widehat{M}$ is torsion-free as 
$kL/P.kL \otimes_{k\zeta H/P} K$-module. 
As $M$ is reduced it is by definition $k\zeta H/P$-torsion free.  Hence the natural 
$kH/P.kH$-linear map $\mu: M \rightarrow \widehat{M}$ via $y \mapsto y \otimes 1$ is a monomorphism.   
Therefore $M$ must be $kL/P.kL$-torsion free. This shows part(i).
\end{proof} 

As a consequence we obtain  

\begin{theorem_2} 
Let $H$ be a finitely generated  
torsion-free nilpotent group of class $2$ . 
For a finitely generated  
$kH$-module $N$ which is  
reduced with annihilator $P$ and  
is finitely generated over $kL$ for some  
abelian subgroup $L \le H$, $N$ is torsion-free as  $kL/P.kL$-module. Moreover if  
$\zeta H$ is cyclic then $N$ has finite length.          
\end{theorem_2}

\end{document}